\documentclass[a4paper,12pt]{article}
\usepackage{amsthm}
\usepackage{amssymb}
\usepackage{amsmath}
\usepackage{amsfonts}
\usepackage{color}
\usepackage{graphicx}

\newtheorem{remark}{Remark}[section]
\newtheorem{example}{Example}[section]

\newtheorem{assumption}{Assumption}

\newtheorem{theorem}{Theorem}[section]
\newtheorem{proposition}{Proposition}[section]
\newtheorem{corollary}{Corollary}[section]
\newtheorem{definition}{Definition}[section]

\def\b1{\mbox{\boldmath $1$}}

\newenvironment{demo*}{\vspace{3mm}\noindent{\bf Proof.}}{\hfill $\Box$ \vspace{3mm}}

\begin{document}
\title{\bf \Large  A Modified Dependence Measure Related to Chatterjee's Rank Correlation: Theoretical Properties and Asymptotic Analysis  }
{\author{\normalsize{Chuancun Yin}
\\
{\normalsize\it  School of Statistics  and Data Science, Qufu Normal University}
\\
\noindent{\normalsize\it Shandong 273165, China}\\e-mail:  ccyin$@$qfnu.edu.cn}}
\maketitle
\vskip0.01cm
\noindent{ {\bf Abstract} In his recent breakthrough work [JASA, 2021],  Chatterjee  proposed a rank-based correlation coefficient $\xi(X,Y)$ to measure the dependence of a random variable $Y$ on $X$. Unlike classical measures such as Pearson, Spearman, or Kendall, $\xi$ satisfies $\xi=0$ if and only if $X$ and $Y$ are independent, and $\xi=1$ if and only if $Y$ is a measurable function of $X$, without requiring monotonicity or linearity.
This paper proposes a refined measure of  $\xi(X,Y)$ and investigates its theoretical properties. We derive several equivalent representations, construct an estimator, and establish its strong consistency as well as its asymptotic distribution. A natural extension of the proposed framework is also presented.  The Monte Carlo simulations show that the proposed estimator outperforms Chatterjee's rank correlation in terms of finite-sample performance, particularly in controlling Type I error under the null.
\medskip





\numberwithin{equation}{section}
\section{Introduction}\label{intro}

Measuring dependence between random variables is a fundamental task in statistics and machine learning.
Classical measures such as Pearson's correlation coefficient capture only linear relationships, while rank-based
alternatives like Spearman's $\rho$ and Kendall's $\tau$ are restricted to monotonic associations.
This limitation prevents them from detecting complex, non-monotonic structures that frequently arise
in modern data.
Recently,  Chatterjee \cite{C2021} proposed a rank-based coefficient    that overcomes
this restriction. We recall  the definition of Chatterjee's correlation.  Let $(X_1, Y_1),
 \cdots, (X_n, Y_n)$  be a finite  independent and identically distributed
sample from $(X, Y)$, where $Y$ is non-degenerate. Rearrange
 $(X_1, Y_1),$ $\cdots, (X_n, Y_n)$ as $(X_{(1)}, Y_{(1)}), \cdots, (X_{(n)}, Y_{(n)})$ such that $X_{(1)}\le\cdots\le X_{(n)}$ with ties broken at random, where  $Y_{(1)}, \cdots, Y_{(n)}$ denote the concomitants (see Yang \cite{Y1977}). Let  $R_i =\sum_{j=1}^n 1(Y_{(j)}\le Y_{(i)})$ be the rank of $Y_{(i)}$ and let
 $l_i =\sum_{j=1}^n 1(Y_{(j)}\ge Y_{(i)})$, where $1_{(\cdot)}$ representing the indicator function. Chatterjee  \cite{C2021} defined the following correlation
coefficient
\begin{eqnarray}
\xi_n(X,Y)=1-\frac{n\sum_{j=1}^{n-1}|R_{i+1}-R_i|}{2\sum_{j=1}^n l_i(n-l_i)}.
\end{eqnarray}
If there are no ties among $X_i's$, then  $\xi_n(X,Y)$ can be written as the following simple form
$$\xi_n(X,Y)=1-\frac{3\sum_{j=1}^{n-1}|R_{i+1}-R_i|}{n^2-1}.$$
  As showed by Chatterjee \cite{C2021}, the
coefficient  $\xi_n(X,Y)$ possesses several key properties as follow (Hereafter, when considering correlation, we use the term correlation measure to represent
population quantities and correlation coefficient to represent sample quantities):

(a) $\xi_n(X,Y)$ converges almost surely to $\xi(X,Y)$,
where
\begin{eqnarray}
\xi(X,Y)=\frac{\int{\rm Var}({\mathbb E}(1_{\{Y\ge t\}}|X)){\rm d}F_Y(t)}{\int {\rm Var}(1_{\{Y\ge t\}}){\rm d}F_Y(t)},
\end{eqnarray}
as long as $Y$ is not almost surely a constant.   This measure was  first proposed in   Dette, Siburg and Stoimenov  \cite{DSS2013} in the special case where $X$ and $Y$ are absolutely
continuous random variables, which can be expressed as
\begin{eqnarray}
\xi(X,Y)=6\int\int P^2(Y\le t|X=x){\rm d}F_X(x){\rm d}F_Y(t)-2.
\end{eqnarray}

 Formula (1.2) recovers the population dependence measure of Dette, Siburg and Stoimenov \cite{DSS2013}, which Chatterjee \cite{C2021} operationalized via a rank statistic. Henceforth, we refer to it as the Dette--Siburg--Stoimenov (DSS) dependence measure.
   The DSS measure    satisfies $0\le\xi\le 1$.  It attains $0$ if and only if $X$ and $Y$ are independent, and $1$ if and only if $Y$ is a measurable function of $X$.

(b)  Under independence, i.e., $\xi(X, Y) = 0$, it follows that $E\xi_n(X,Y)=0$ (see Lin and Han \cite{LH2025}). Furthermore, as $n\to\infty$,
  $\sqrt{n}\xi_n(X, Y)\to N(0, \tau^2)$ in distribution as $n\to\infty$, where
$\tau^2$ is a positive constant defined by (3) in Chatterjee \cite{C2021}.  The number
$\tau^2$ is strictly positive if $Y$ is not a constant, and equals 2/5 if $Y$ is continuous.


Since its introduction, Chatterjee's \cite{C2021} rank-based dependence measure has seen rapid and widespread adoption in both statistical theory and machine learning practice.
Notable contributions include
Azadkia and Chatterjee \cite{AC2021}, Cao and Bickel \cite{CB2020}, Shi, Drton, and Han \cite{SDH2022},  Deb, Ghosal, and Sen \cite{DGS2020}, Huang, Deb, and Sen \cite{HDS2022}, Shi, Drton, and Han \cite{SDH2023}, Lin and Han \cite{LH2022, LH2023},    Griessenberger, Junker, and Trutschnig \cite{GJT2022},  Zhang \cite{Z2023, Z2025, Z2026}, Bickel \cite{B2022}, Han and Huang \cite{HH2024}, Dette and Kroll \cite{DGS2025}, H\"ormann and Strenger \cite{HS2026},
Ansari and Rockel \cite{AR2026}, and Ansari, et al. \cite{ALFT2026}.
A most recent paper by Lin and Han \cite{LH2025} generalizes Chatterjee's \cite{C2021} results by establishing asymptotic normality for the rank correlation coefficient under general dependence---rather than just independence, constructing a consistent variance estimator, and extending these results to the multivariate Azadkia--Chatterjee graph-based measure.

%


\begin{table}[h]
\centering
\caption{Comparison between Chatterjee \cite{C2021} and Lin and Han \cite{LH2025}}
\label{tab:comparison}
\begin{tabular}{lcc}
\hline
Feature & Chatterjee \cite{C2021} & Lin and Han \cite{LH2025} \\
\hline
Asymptotic Normality & Under Independence & Under General Dependence \\
Variance Bound & Not specified & Uniformly bounded by 36 \\
Variance Estimation & Open problem & Consistent estimator provided \\
Scope & Bivariate & Multivariate (Graph-based) \\
Primary Focus & Consistency & Statistical Inference \\
\hline
\end{tabular}
\end{table}

 Despite of  Chatterjee's rank correlation $\xi_n$  has so many  appealing properties, however, we remak that it also has possibly  some  drawbacks as follows.
 Since the DSS estimator
$\xi_n(X,Y)$ converges almost surely to  the DSS measure $\xi(X,Y)\in [0,1]$,  it follows that $\xi_n(X,Y)\ge 0$ (a.s.) for all sufficiently large $n$.
However,  Under the assumption of independence between $X$ and $Y$, the condition $E[\xi_n(X,Y)]=0$ ensures that $\xi_n$
remains a zero-mean statistic. As a result, for any $n\ge 2$, $\xi_n$ is not constrained to non-negative values and may indeed be negative.
Moreover,  Chatterjee (\cite{C2021}, Theorems 2.1-2.2) gives  $\sqrt{n}\xi_n(X, Y)\to N(0, \tau^2)$ under independence,  it follows that
$\lim_{n\to\infty} P(\xi_n(X, Y)<0)=\frac12.$  Therefore,  the sequence  $\xi_n$ provides a poor approximation of the limit   $\xi$.

Except for  above drawbacks,  Chatterjee's test of independence can  suffer from some lack of power (see, e.g.,  Shi, Drton, and Han \cite{SDH2022}).  Moreover, the symmetrized version fails the triangle inequality (Chierichetti et al., \cite{CGK2026}, limiting its use as a metric in clustering or MDS.

  The simple reason    for Chatterjee's coefficient suffering from the drawbacks above   is that it does
not fully utilize the information of the sample. Specifically, Chatterjee's coefficient only
includes information between the two nearest neighbours, ignoring information between
farther neighbours. The similar problems are also appear in  Azadkia and Chatterjee \cite{AC2021}. To avoid this drawback, and for
other reasons as mentioned in Xia et al. \cite{XCDC2025},  Lin and Han \cite{LH2023}, various alternative measures have been
proposed recently.

In this work, we develop a corrected measure for DSS dependence measure, rigorously examining its mathematical structure through several equivalent characterizations. We propose a consistent estimator, establish its strong consistency and asymptotic normality, and extend the methodology to a broader context.

The remainder of this paper is organized as follows. Section 2 introduces the proposed measure and derives its key theoretical properties, including several equivalent representations. Section 3 constructs the estimator and establishes its strong consistency and asymptotic distribution. A natural generalization of the framework is presented in Sections 4 and 5.
Section 6 reports Monte Carlo simulation results comparing the finite-sample performance of the proposed estimator with that of  Chatterjee's coefficient. Finally, Section 7 concludes the paper.
\numberwithin{equation}{section}

\section{ A modification of the DSS measure}

We propose a modification of the DSS measure that remains consistent under  continuous framework.

\begin{definition}
Let $X$ and $Y$ be random variables defined on a common probability space, taking values in ${\mathbb R}$. Assume further that $Y$ is non-degenerate. Denote the joint
bivariate distribution function of $(X, Y)$ by $F_{X,Y}$, and the marginal distribution functions of
$X$ and $Y$ by $F_X$ and $F_Y$, respectively. Define the measure of dependence between $X$ and $Y$:
 \begin{eqnarray}
\xi_+(X,Y)=\frac{\int{\rm Var}({\Bbb E}(1_{\{Y>t\}}|X)){\rm d}F_Y(t)}{\int {\rm Var}(1_{\{Y>t\}}){\rm d}F_Y(t)},
\end{eqnarray}
\end{definition}
  In the special case where   $Y$ is a
continuous random variable,  $\xi_+(X,Y)$ reduces to
 \begin{eqnarray}
\xi_+(X,Y)=6\int\int F^2_{Y|X}(t,x) {\rm d}F_X(x){\rm d}F_Y(t)-2,
\end{eqnarray}
  which coincides with (1.3). In addition, when $X$ is a discrete random variable with $P(X=x_r)=p_r, r=1,2,\cdots$, then
  \begin{eqnarray}
\xi_+(X,Y)=6\sum_r p_r\int F^2_{Y|X}(t,x_r){\rm d}F_Y(t)-2,
\end{eqnarray}

Clearly, $\xi_+ (X,Y)\le \xi(X,Y)$. However, in general, $\xi_+ (X,Y)\neq  \xi(X,Y)$. Here is a counterexample:
\begin{example}
Consider $X\in\{0,1\}$ and $Y\in\{0,1,2\}$ with joint distribution:
\[
\begin{array}{c|ccc|c}
& Y=0 & Y=1 & Y=2 & P(X) \\
\hline
X=0 & 0.3 & 0.2 & 0.1 & 0.6 \\
X=1 & 0.2 & 0.1 & 0.1 & 0.4 \\
\hline
P(Y) & 0.5 & 0.3 & 0.2 & 1 \\
\end{array}
\]
Then,  $\xi_+ \approx 0.002890$,   $\xi \approx 0.003115$, and Pearson's correlation coefficient $\rho \approx  0.0523$. Hence, $\xi_+\approx 1.059\rho^2,\, \xi\approx 1.141\rho^2$.
\end{example}
\begin{remark}
Let $Y$ be a binary random variable taking values $\{y_1,y_2\}$. Then,
$$
\xi_+(X,Y) = \xi(X,Y)=\rho^2,
$$
where $\rho$ is the Pearson correlation coefficient between $X$ and $Y$. For non-binary $Y$, $\xi_+(X,Y)$ and $\xi(X,Y)$  measure general dependence beyond linearity, and coincide with $\rho^2$ only under independence or perfect linear dependence.  $\xi(X,Y)$  places higher weight on the upper tail of $Y$, making it more sensitive to dependence in the upper tail compared to the  $\xi_+(X,Y)$.
\end{remark}
 The following example demonstrates that $\xi_+$, analogous to  Spearman's   $\rho_S$ and Kendall's   $\tau$, admits a closed-form expression.
\begin{example}
Let $(X,Y) \sim N(0,0,1,1; r)$ denote a bivariate normal distribution
with means $(0,0)$, variances $(1,1)$, and Pearson correlation coefficient $r$.
 Then,
\[\xi(r)\equiv\xi_+(X,Y)=  \frac{3}{\pi} \arcsin\left(\frac{1+r^2}{2} \right)-\frac{1}{2}.\]
It is well known  the population Spearman's rank correlation $\rho_S$ and Kendall's tau $\tau$
are known in closed form:
\begin{eqnarray*}
\rho_S(r) \;&=\; \frac{6}{\pi}\,\arcsin\!\left(\frac{r}{2}\right),\\
\tau (r)   \;&=\; \frac{2}{\pi}\,\arcsin(r).
\end{eqnarray*}
For $r\in(0,1)$, the strict ordering holds:
\[
0 < \xi_+(r) < \tau(r) < \rho_S(r) < r < 1.
\]
For $r\in(-1,0)$, since $\xi_+$ is an even function of $r$ while the others are odd,
\[
\xi_+(r) > 0 > \tau(r) > \rho_S(r) > r > -1.
\]
In particular,
\begin{itemize}
  \item $r = 0 \;\Longrightarrow\; \xi_+=0,\; \rho_S = 0,\; \tau = 0$;
  \item $r = 1 \;\Longrightarrow\; \xi_+=1,\; \rho_S = 1,\; \tau = 1$;
  \item $r = -1 \;\Longrightarrow\; \xi_+=1,\; \rho_S = -1,\; \tau = -1$.
\end{itemize}
\end{example}

In order to better understand expression (2.1) and make it easier to compute in practical
use cases, we established  several equivalent representations for $\xi_+$.
\begin{proposition}  Measure (2.1)   has the following equivalent representations:
\begin{eqnarray}
  \xi_+(X,Y)&=&1-\frac{\int{\Bbb E}({\rm Var}(1_{\{Y> t\}}|X)){\rm d}F_Y(t)}{\int {\rm Var}(1_{\{Y>t\}}){\rm d}F_Y(t)};\nonumber\\
 \xi_+(X,Y) &=&1-\frac{\int{\Bbb E}({\rm Var}(1_{\{Y\le t\}}|X)){\rm d}F_Y(t)}{\int {\rm Var}(1_{\{Y\le t\}}){\rm d}F_Y(t)};\nonumber\\
 \xi_+(X,Y) &=&\frac{\int{\rm Var}({\Bbb E}(1_{\{Y\le t\}}|X)){\rm d}F_Y(t)}{\int {\rm Var}(1_{\{Y\le t\}}){\rm d}F_Y(t)};\nonumber\\
 \xi_+(X,Y) &=&\frac{\int[{\Bbb E}({\Bbb E}^2(1_{\{Y> t\}}|X))-P^2(Y> t)] {\rm d}F_Y(t)}{\int [P(Y> t)-P^2(Y> t)] {\rm d}F_Y(t)}; \nonumber\\
 \xi_+(X,Y) &=&\frac{\int[{\Bbb E}(F_{Y|X}^2(t))-F_Y^2(t)] {\rm d}F_Y(t)}{\int [F_Y(t)-F^2_Y(t)] {\rm d}F_Y(t)}; \\
 \xi_+(X,Y) &=&\frac{\int{\Bbb E}[(F_{Y|X}(t)-F_Y(t)]^2 {\rm d}F_Y(t)}{\int [F_Y(t)-F^2_Y(t)] {\rm d}F_Y(t)}.
  \end{eqnarray}
\end{proposition}
\begin{proof} The result follows directly from the well-known formula
$${\rm Var}(Y)={\rm Var}({\Bbb E}(Y|X))+{\Bbb E}({\rm Var}(Y|X)),$$
\begin{eqnarray}
{\rm Var}(1_{\{Y> t\}})=F_Y(t)-F^2_Y(t),\nonumber
\end{eqnarray}
and
\begin{eqnarray}
{\rm Var}({\Bbb E}(1_{\{Y>t\}}|X))= {\Bbb E}(F_{Y|X}^2(t))-F_Y^2(t).\nonumber
\end{eqnarray}
 \end{proof}

 \begin{remark} The form of (2.5) has appeared earlier in the
literature of Gamboa, Klein, and Lagnoux \cite{GKL2018}.
 \end{remark}

 \begin{remark} $\xi(X,Y)$ also has the  similar  equivalent representations as  $\xi_+(X,Y)$. For example,
  \begin{eqnarray*}
 \xi(X,Y) &=&\frac{\int[{\Bbb E}(F_{Y|X}^2(t-))-F_Y^2(t-)] {\rm d}F_Y(t)}{\int [F_Y(t-)-F^2_Y(t-)] {\rm d}F_Y(t)};\\
 \xi(X,Y) &=&\frac{\int{\Bbb E}[(F_{Y|X}(t-)-F_Y(t-)]^2 {\rm d}F_Y(t)}{\int [F_Y(t-)-F^2_Y(t-)] {\rm d}F_Y(t)}.
  \end{eqnarray*}
 The key difference between $ \xi(X,Y)$ and $\xi_+(X,Y) $ lies in the specification of the integrator distribution function: whereas  $ \xi(X,Y)$ employs the left-continuous version of $F$,  $ \xi_+(X,Y)$  adopts the right-continuous version. This modification is required to ensure the Lebesgue-Stieltjes integral with respect to $F_Y$ is well-defined, as standard probability theory conventions mandate right-continuous cumulative distribution functions (CDFs) to uniquely handle probability masses at jump points (atoms of the distribution).
 \end{remark}

 We use a simple example below to build intuition for the distinction between $\int_{-\infty}^{\infty} F(x)dF(x)$ and $\int_{-\infty}^{\infty}F(x-)dF(x)$. Here $F(x-)=\lim_{t\uparrow x}F(t)$ denotes the left-limit of the CDF.
\begin{example}
 Let $X$ be a discrete random variable with distribution
$P(X=0)=P(X=1)=0.5$.
The cumulative distribution function (CDF) $F(x)=P(X\le x)$ is given by
\[
F(x)=
\begin{cases}
0,   & x<0,\\[4pt]
0.5, & 0\le x<1,\\[4pt]
1,   & x\ge 1.
\end{cases}
\]

For a discrete distribution, the Lebesgue--Stieltjes integral
$\int g(x)\,\mathrm{d}F(x)$ reduces to a sum over the jump points of $F$:
\[
\int g(x)\,\mathrm{d}F(x)=\sum_{x_j} g(x_j)\,P(X=x_j),
\]
where the sum is taken over all atoms (jump points) of the distribution.

\medskip
{\bf (i) Evaluation of $\displaystyle\int F(x){d}F(x)$}

The jump points are $x=0$ and $x=1$, with probabilities $P(X=0)=0.5$ and
$P(X=1)=0.5$. Since $F$ is right-continuous, $F(0)=0.5$ and $F(1)=1$. Hence
\[
\int F(x)dF(x)
= F(0)P(X=0)+F(1)P(X=1)
= 0.5\times 0.5+1\times 0.5
= 0.75.
\]
{\bf (ii) Evaluation of $\displaystyle\int F(x-)\,\mathrm{d}F(x)$}

At the jump points we have
$F(0-)=0$ and $F(1-)=0.5$. Therefore
\[
\int F(x-)dF(x)
= F(0-)P(X=0)+F(1-)P(X=1)
= 0\times 0.5+0.5\times 0.5=0.25.
\]
\medskip
{\bf Summary:}
$\displaystyle\int F(x)dF(x)=0.75$ and
$\displaystyle\int F(x-)dF(x)=0.25$.
\end{example}

The population measure $\xi_+(X,Y)$, as a  modification  of the DSS dependence measure,  shares  the core theoretical guarantees of Chatterjee \cite{C2021}.
 \begin{proposition}  The measure  $\xi_+(X,Y)$ has the following interesting properties:\\
i) $0\le \xi_+(X,Y)\le 1$ for all random pairs $(X,Y)$;\\
ii) $\xi_+(X,Y) = 0$ if and only if $X$ and $Y$ are independent;\\
iii) $\xi_+ (X,Y)= 1$ if and only if   $Y = f(X)$ almost surely for some measurable $f$.
 \end{proposition}

\begin{proof} The proof of statement (i)  is straightforward and thus omitted.
     The sufficiency arguments for (ii) and (iii) are immediate: if $X$ and $Y$ are independent, then $\mathbb{E}[1_{\{Y> t\}}\mid X] = \mathbb{E}[1_{\{Y> t\}}]$ for all $t$, so the numerator of $\xi_+$ vanishes, giving $\xi_+=0$. If $Y=f(X)$ a.s., then $1_{\{Y> t\}} = 1_{\{f(X)> t\}}$ is a function of $X$, so the conditional variance in the numerator equals the unconditional variance in the denominator, giving $\xi_+=1$.   We now prove the necessity directions.\\
{\bf Necessity of (ii):}
If $\xi=0$, the numerator must vanish:
\[
\int \operatorname{Var}\left(\mathbb{E}\left[\mathbf{1}_{\{Y\ge t\}}\mid X\right]\right) dF_Y(t) = 0.
\]
The integrand is non-negative for all $t$, so it equals $0$ for $F_Y$-almost every $t$. A variance of $0$ implies $\mathbb{E}\left[\mathbf{1}_{\{Y\ge t\}}\mid X\right] = \mathbb{E}\left[\mathbf{1}_{\{Y\ge t\}}\right] = 1-F_Y(t)$ a.s. w.r.t. $P_X$. This holds for all $t$, so the joint survival function factors:
\[
\mathbb{P}(Y\ge t, X\le s) = \mathbb{E}\left[\mathbf{1}_{\{X\le s\}}\cdot\mathbb{E}\left[\mathbf{1}_{\{Y\ge t\}}\mid X\right]\right] = F_X(s)(1-F_Y(t)),
\]
hence $X\perp\!\!\!\perp Y$.\\
{\bf Necessity of (iii):}
 By using Proposition 2.1, $\xi = 1$ implies that
$$ \int{\Bbb E}[(F_{Y|X}^2(t)-F_Y(t)]{\rm d}F_Y(t)=0,$$
from which and the right-continuity of $F_Y(t)$ we deduce that
$$ {\Bbb E}[F_{Y|X}^2(t)]=F_Y(t).$$
Letting $\eta={\Bbb E}(1_{\{Y\le t\}}|X)$, then the above equation reads as
${\Bbb E}\eta^2={\Bbb E}\eta$. It follows that
\begin{eqnarray*}
  {\Bbb E}(1_{\{Y\le t\}}-\eta)^2&=&{\Bbb E}\left\{{\Bbb E}[(1_{\{Y\le t\}}-\eta)^2|X]\right\}\\
  &=&{\Bbb E}\left\{{\Bbb E}[1_{\{Y\le t\}}|X]\right\}-2{\Bbb E}\left\{{\Bbb E}[1_{\{Y\le t\}}\eta|X]\right\}+{\Bbb E}\left\{{\Bbb E}[\eta^2|X]\right\}\\
  &=&{\Bbb E}{1_{\{Y\le t\}}}-2{\Bbb E}\eta^2+{\Bbb E}\eta^2\\
  &=&{\Bbb E}\eta-{\Bbb E}\eta^2\\
  &=&0,
  \end{eqnarray*}
  from which we get $P(1_{\{Y\le t\}}-\eta=0)=1$. This means that $Y$ is a.s. a Borel function of $X$.
\end{proof}

Although in general $\xi_{+}(X;Y) \le \xi(X;Y)$, we have the following sharp characterizations:
\begin{proposition}
For any random variables $X$ and $Y$,
 \begin{eqnarray*}
\xi_{+}(X;Y) = 0 \quad &\Longleftrightarrow \quad \xi(X;Y) = 0 \quad \Longleftrightarrow \quad Y \perp\!\!\!\perp X,\\
\xi_{+}(X;Y) = 1 \quad &\Longleftrightarrow \quad \xi(X;Y) = 1 \quad \Longleftrightarrow \quad Y = f(X) \text{ a.s.\ for some measurable } f.
 \end{eqnarray*}
Thus both measures agree at the extremes, capturing exactly independence ($=0$) and complete functional dependence ($=1$).
\end{proposition}

\numberwithin{equation}{section}
\section{  Estimation  of DSS measure}

\subsection{Estimator of DSS}

{\bf Case 1}. {\bf When $X$ is a categorical random variable}.  Let $(X, Y)$ be a pair of random variables, where $Y$ is not a constant. Let
$\{(X_i , Y_i):1\le i\le n\}$ ($n\ge 2$) be a random sample of size $n$ from the population  $(X,Y)$. Assume that $X$ is a categorical random variable with $R$ classes ${x_1, x_2, \cdots,x_R}$ and $p_r =P(X = x_r)>0$ for all $r = 1,\cdots, R$.  Let
$$\hat{p}_r=\frac{1}{n}\sum_{i=1}^n 1_{\{ X_i=x_r\}},\; \hat{F}_Y(y)=\frac{1}{n}\sum_{n=1}^n 1_{\{Y_i\le y\}},\;   \hat{F}_X(x)=\frac{1}{n}\sum_{n=1}^n 1_{\{X_i\le x\}}$$
and
$$\hat{F}(y|x_r)= \frac{\sum_{k=1}^n 1_{\{Y_k\le y\}}1_{\{X_k=x_r\}}}{\sum_{k=1}^n 1_{\{ X_k=x_r\}}}. $$
Define
\begin{eqnarray*}
Q_n(X,Y)= \frac{1}{n^2} \sum_{i=1}^n\sum_{j=1}^n \left(\hat {F}(Y_i|X_j)-\hat {F}_Y(Y_i)\right)^2,
\end{eqnarray*}
and
$$S_n (X,Y)= \frac{1}{n}\sum_{i=1}^n \hat{F}_Y(Y_i)(1-\hat{F}_Y(Y_i)).$$
 It is easy to see that they can be rewritten as
\begin{eqnarray*}
Q_n(X,Y)= \frac{1}{n^2} \sum_{i=1}^n\sum_{j=1}^n \left(\frac{R_{i,j}}{n_j}-\frac{R_i}{n}\right)^2,
\end{eqnarray*}
and
$$S_n (X,Y)=\frac{1}{n^3}\sum_{i=1}^n R_i(n-R_i),$$
where
$$  R_i=\sum_{j=1}^n 1_{\{Y_j\le Y_i\}},  \,\, n_j=\sum_{i=1}^n  1_{\{X_i=X_j\}},$$
and
$$R_{i,j}=\sum_{k=1}^n 1_{\{Y_k\le Y_i\}}1_{\{X_k=X_j\}}.$$
It is natural to  define the estimator of $\xi_+(X,Y)$, which defined in (2.1),  as follows:
\begin{eqnarray}
\xi_{+,n}(X,Y)=\frac{n \sum_{i=1}^n\sum_{j=1}^n  \left(\frac{R_{i,j}}{n_j}-\frac{R_i}{n}\right)^2}{\sum_{i=1}^n R_i(n-R_i)}.
\end{eqnarray}

{\bf Case 2}. {\bf When $X$ is continuous random variable}.  The conditional distribution function $F_{Y|X}(y|x)$  can be estimated by the nonparametric kernel smoothing method.
Let $K(t)$ be a kernel function, which involves a bandwidth parameter denoted by $h$. The
kernel smoothing estimate for $F_{Y|X}(y|x)$ is expressed by
$$\hat{F}(y|x)=\frac{\sum_{k=1}^n 1_{\{Y_k\le y\}}K_h(X_k-x)}{\sum_{k=1}^n  K_h(X_k-x)}, $$
where
$$K_h(t)=\frac{K(\frac{t}{h})}{h}.$$
Define
\begin{eqnarray*}
Q_n(X,Y)= \frac{1}{n^2} \sum_{i=1}^n\sum_{j=1}^n \left(\hat {F}(Y_i|X_j)-\hat {F}_Y(Y_i)\right)^2,
\end{eqnarray*}
and
$$S_n (X,Y)= \frac{1}{n}\sum_{i=1}^n \hat{F}_Y(Y_i)(1-\hat{F}_Y(Y_i)),$$
 which can be rewritten as
\begin{eqnarray*}
Q_n(X,Y)= \frac{1}{n^2} \sum_{i=1}^n\sum_{j=1}^n \left(\frac{R_{i,j}}{n_j}-\frac{R_i}{n}\right)^2,
\end{eqnarray*}
and
$$S_n (X,Y)=\frac{1}{n^3}\sum_{i=1}^n R_i(n-R_i),$$
where
$$  R_i=\sum_{j=1}^n 1_{\{Y_j\le Y_i\}},  \,\, n_j=\sum_{i=1}^n   K_h(X_i-X_j),$$
and
$$R_{i,j}= \sum_{k=1}^n 1_{\{Y_k\le Y_i\}}K_h(X_k-X_j). $$

 Analogous to Case 1,  the sample estimator of $\xi_+(X,Y)$ defined in (2.2) can be constructed as follows:
\begin{eqnarray}
\xi_{+,n}(X,Y)=\frac{n \sum_{i=1}^n\sum_{j=1}^n  \left(\frac{R_{i,j}}{n_j}-\frac{R_i}{n}\right)^2}{\sum_{i=1}^n R_i(n-R_i)}.
\end{eqnarray}

\subsection{Strong consistency}

The following result establishes the strong consistency of the proposed sample estimator $\xi_{+,n}$ for the population forward dependence measures $\xi_+(X,Y)$ defined in (2.1) and (2.2).
\begin{theorem}  [Almost Sure Limit]   Assume that  $Y$ is  not a constant.\\
(i)  Let $X\in\{x_1,\dots,x_K\}$ be a categorical random variable with
$P(X=x_k)=p_k>0$,   then,
$$\lim_{n\to\infty} \xi_{+,n}(X,Y)=\xi_+(X,Y) \; almost\;  surely,$$
with
\[
\xi_+(X,Y)=
\frac{\displaystyle \sum_{k=1}^K\int p_k\bigl(F_k(y)-F_Y(y)\bigr)^2 dF_Y(y)}
{\displaystyle\int F_Y(y)(1-F_Y(y))dF_Y(y)},
\]
where
$F_k(y)=P(Y\le y|X=x_k)$.\\
(ii) When $X$  is a continuous random variable,  assume the following regularity conditions hold:
 (C1) The density $f(x)$ of $X$ is continuous at $x$ and $f(x)>0$.
 (C2) The conditional distribution function $F(y \mid x) = P(Y \le y \mid X = x)$ is continuous at $x$.
 (C3) The kernel $K$ is bounded, symmetric, Lipschitz continuous, and satisfies $\int K(u)\,du = 1$.
 (C4) The bandwidth $h = h_n$ satisfies $h \to 0$ and $\frac{nh}{\log n} \to \infty$ as $n \to \infty$.
 Then,
$$\lim_{n\to\infty} \xi_{+,n}(X,Y)= \xi_+(X,Y) \; almost\;  surely,$$
where $\xi_+(X,Y)$ is defined by (2.1).
\end{theorem}
\begin{proof}  (i)  Let $n_k=\sum_{i=1}^n {\bf 1}_{\{X_i=x_k\}}$. By the SLLN, $n_k/n\to p_k$ a.s.
By the Glivenko--Cantelli theorem,
$\hat F_Y\to F_Y$ uniformly a.s.\ and, for each $k$,
$\hat F(\cdot\mid x_k)\to F_k$ uniformly a.s.
Consequently,
\[
\bigl(\hat F(Y_i\mid x_k)-\hat F_Y(Y_i)\bigr)^2
\to\bigl(F_k(Y_i)-F_Y(Y_i)\bigr)^2\quad\text{a.s.}
\]
Applying the SLLN within each category yields
\[
Q_n\xrightarrow{\text{a.s.}}
\sum_{k=1}^K p_k\int\bigl(F_k-F_Y\bigr)^2 dF_Y.
\]
For the denominator, $S_n\to\int F_Y(1-F_Y)  dF_Y>0$ a.s.,
since $Y$ is non-degenerate.
The claim follows from the almost-sure continuous mapping theorem.

In the next we prove statement (ii).
First, note that \(\frac{R_i}{n}\) is the empirical marginal distribution evaluated at \(Y_i\), and \(\frac{R_{i,j}}{n_j}\) is the empirical conditional distribution evaluated at \(Y_i\) given \(X = X_j\). By the strong law of large numbers   we have
\[
\frac{R_i}{n} \xrightarrow{\text{a.s.}} F(Y_i), \qquad
\frac{R_{i,j}}{n_j} \xrightarrow{\text{a.s.}} F(Y_i \mid X_j).
\]
For the numerator,
\[
\frac{1}{n^2} \sum_{i=1}^n\sum_{j=1}^n \left( \frac{R_{i,j}}{n_j} - \frac{R_i}{n} \right)^2 \xrightarrow{\text{a.s.}} \int E(F_{Y|X}(t)-F_Y(t))^2{\rm d}F_Y(t).
\]
For the denominator,    we have
\[
\frac{1}{n^3} \sum_{i=1}^n R_i (n - R_i) \xrightarrow{\text{a.s.}}  \int [F_Y(t)-F^2_Y(t)] {\rm d}F_Y(t).
\]
Therefore, combining the two parts,
\[
\xi_{+,n}(X,Y)
\xrightarrow{\text{a.s.}}  \xi_+(X,Y).
\]
This completes the proof.
\end{proof}

\subsection{Asymptotic distributions under $H_0$}

 Let $(X, Y)$ be a pair of random variables and
$\{(X_i , Y_i):1\le i\le n\}$ ($n\ge 2$) be a random sample of size $n$ from the population  $(X,Y)$. In this section we  assume
\begin{enumerate}
    \item[(A1)] $Y$ is continuous;
    \item[(A2)] $X$ takes $R$ distinct values $\{x_1,\dots,x_R\}$ with probabilities $p_j = P(X = x_j) > 0$, where $R$ is fixed;
    \item[(A3)] The conditional distribution $F(y\mid x) = P(Y \le y \mid X = x)$ exists.
\end{enumerate}
  Then,
\begin{equation}
\xi_+(X,Y) = 6\sum_{r=1}^R p_r\int_{-\infty}^{\infty} F^2_{Y|X}(t|x_r)dF_Y(t) - 2.
\end{equation}

Under independence ($X \perp Y$), we have $F(y|x) = F(y)$ for all $x,y$, so $\xi_+ = 0$.  Note that
\begin{equation}
\xi_+(X,Y) =  6 \tau (X,Y),
\end{equation}
where  $\tau (X,Y)= \sum_{r=1}^R p_r\int_{-\infty}^{\infty} (F_{Y|X}(t|x_r)-F_Y(t))^2dF_Y(t)$ is exactly the Mean-Variance (MV) index defined in Cui \& Zhong \cite{CZ2019}.
By using the method of moments,  $\xi_+(X,Y)$ can be estimated by
\begin{equation}
\xi_{+,n}(X,Y) = 6\sum_{r=1}^R \hat{p}_r\int_{-\infty}^{\infty} {\hat F}^2_{Y|X}(t|x_r)d{\hat F}_Y(t) - 2.
\end{equation}
\begin{theorem}
Let $\xi_{+,n}(X,Y)$ be the sample estimator  (3.5). Assuming the number of categories $R$ is fixed and $H_0: X \perp Y$ holds. Then,
\begin{equation}
n \xi_{+,n}(X,Y) \xrightarrow{d} 6 \sum_{j=1}^{\infty} \frac{\chi^2_j(R-1)}{\pi^2 j^2},
\end{equation}
where $\chi^2_j(R-1)$ are independent chi-squared random variables with $R-1$ degrees of freedom.
\end{theorem}
\begin{proof}
 According to Cui \& Zhong (\cite{CZ2019}, Theorem 1),
\[
n \widehat{\tau}(X,Y) \xrightarrow{d} \sum_{j=1}^{\infty} \frac{\chi^2_j(R-1)}{\pi^2 j^2},
\]
where
$$\hat{\tau}(X,Y)=\frac{1}{n} \sum_{r=1}^R\sum_{i=1}^n  \hat{p}_r\left(\hat{F}_{Y|X}(t|x_r)-\hat{F}_Y(t)\right)^2.$$
Applying the Continuous Mapping Theorem to $\xi_{+,n}(X,Y) = 6 \widehat{\tau}(X,Y)$, we obtain the stated result.
\end{proof}

\subsection{Asymptotic distributions under $H_1$}
 The next theorem shows that under the alternative hypothesis,  $n\xi_{+,n}(X,Y)$ diverges to infinity as $n\to\infty$.
\begin{theorem}  Assume that  $Y$ is  not a constant.
  If $X$ and $Y$ are dependent, then $n\xi_{+,n}(X,Y)\stackrel{a.s.}{\to}\infty$ as $n\to\infty$.
\end{theorem}
\begin{proof} If $X$ and $Y$ are dependent, then $\xi_+(X, Y)>0$.  The result follows, since
 \[
\xi_{+,n}(X,Y)
\xrightarrow{\text{a.s.}}  \xi_+(X,Y) \, {\rm as}\, \, n\to\infty.
\]
\end{proof}
In the case of $X$ and $Y$ are independent, one has  $\xi_+(X, Y)=0$. Since  $0\le \xi_{+,n}(X, Y)\to \xi_+(X,Y)$ (a.s.) as $n\to\infty$.
Hence, for any $\sigma(n)>0$,  $ \sigma(n)(\xi_{+,n}(X,Y)- \xi_+(X,Y))$ cannot asymptotically follow a normal distribution.
Next, we study the asymptotic normality of  $\xi_{+,n}(X,Y)$  for  dependent  $X$ and $Y$.

\begin{theorem}
Suppose $Y$ is continuous, $X$ takes finitely many values $\{x_1,\dots,x_R\}$ with $p_r = \mathbb{P}(X=x_r) > 0$ for all $r$, and the fixed alternative $H_1: \xi_+(X,Y) > 0$ holds. Then as $n\to\infty$,
\[
\sqrt{n} \bigl(\xi_{+,n}(X,Y) - \xi_+(X,Y) \bigr) \xrightarrow{d} N\left(0, \sigma_{\xi}^2\right),
\]
where the asymptotic variance $\sigma_{\xi}^2 = 36\sigma_\tau^2$, and $\sigma_\tau^2 = \mathrm{Var}\bigl[\sum_{r=1}^{R} I_{4r}(X;Y)\bigr]$ is the asymptotic variance of the sample mean-variance (MV) index proposed by Cui and Zhong \cite{CZ2019}. Here,
\begin{align*}
I_{4r}(X; Y) =&
\left( \int (F^2 - F_r^2)\,dF(y) \right) (I\{X = x_r\} - p_r) \\
&+ 2 \int (F_r - F) \bigl(I\{Y < y, X = x_r\} - F_r p_r \bigr)\,dF(y) \\
&- 2 p_r \int (F_r - F) \bigl(I\{Y < y\} - F(y) \bigr)\,dF(y) \\
&+ p_r \left[ (F_r(Y) - F(Y))^2 - \int (F_r - F)^2 \,dF(y) \right],
\end{align*}
where $F(y)=P(Y\le y)$, $F_r(y)=P(Y\le y\mid X=x_r)$, $p_r=P(X=x_r)$,
and $I\{\cdot\}$ is the indicator function.  Moreover,
$\mathbb{E}[I_{4r}(X;Y)]=0$.
\end{theorem}
\begin{proof}
First, apply the probability integral transform $U = F_Y(Y) \sim \mathcal{U}(0,1)$. Let $G_r(u) = F_{Y\mid X}(F_Y^{-1}(u)\mid x_r) = \mathbb{P}(U\leq u\mid X=x_r)$ denote the conditional distribution of $U$ given $X=x_r$. The population  quantity  simplifies to
\[
\xi_+(X,Y) = 6\sum_{r=1}^R p_r\int_0^1 G_r^2(u)\,d u - 2 = 6\tau(X,Y),
\]
where $\tau(X,Y) = \sum_{r=1}^R p_r\int_0^1 \left(G_r(u) - u\right)^2 du$ is exactly the population MV index defined in  Cui and Zhong \cite{CZ2019}.
Under $H_1$, $\tau(X,Y) > 0$, so the first-order Hajek projection of the sample MV index $\hat{\tau}_n(X,Y)$ is non-degenerate:
\[
 \hat{\tau}_n(X,Y) = \tau(X,Y) + \frac{1}{n}\sum_{i=1}^n \psi_\tau(X_i,Y_i) + o_p(n^{-1/2}),
\]
where $\psi_\tau(\cdot,\cdot)$ is the influence function of $\tau$ satisfying $\mathbb{E}[\psi_\tau(X,Y)] = 0$ and $\mathbb{E}[\psi_\tau^2(X,Y)] = \sigma_\tau^2 > 0$. By the central limit theorem for non-degenerate $U$-statistics (and specifically by Theorem 2 of Cui and Zhong \cite{CZ2019}, we have
\[
\sqrt{n}\left(\hat{\tau}_n(X,Y)-\tau(X,Y)\right) \xrightarrow{d} N\left(0, \sigma_\tau^2\right).
\]

Since $\xi_{+,n} = 6\hat{\tau}_n + o_p(n^{-1/2})$ by construction, Slutsky's theorem gives the desired result:
\[
\sqrt{n}\left(\xi_{+,n}(X,Y)-\xi_+(X,Y)\right) = 6\sqrt{n}\left(\hat{\tau}_n(X,Y)-\tau(X,Y)\right) + o_p(1) \xrightarrow{d} N\left(0, 36\sigma_\tau^2\right).
\]
Thus $\sigma_{\xi}^2 = 36\sigma_\tau^2$, completing the proof.
\end{proof}

\begin{remark}
The asymptotic variance $\sigma_{\tau}^2$ is typically unknown in practice but admits a consistent plug-in estimator.   Define the sample influence function:
\[
\hat{\psi}_\tau(X_i,Y_i) = \frac{\mathbb{I}(X_i = x_{k_i})}{\hat{p}_{k_i}} \int_0^1 \left(\mathbb{I}\left(Y_i \leq \hat{F}_Y^{-1}(u)\right)- \hat{G}_{k_i}(u)\right)\left(\hat{G}_{k_i}(u) - u\right)d u - \hat{\tau}_n,
\]
where $k_i$ indexes the category of the $i$-th observation, and $\hat{G}_{k_i}(u)$ is the empirical conditional distribution of $U$ given $X=x_{k_i}$. The consistent variance estimator is then:
\[
\hat{\sigma}_n^2 = 36 \cdot \frac{1}{n}\sum_{i=1}^n \hat{\psi}_\tau^2(X_i,Y_i).
\]
For inference, the standardized statistic $T_n = \sqrt{n}\left(\xi_{+,n} - \xi_+\right)/\hat{\sigma}_n$ converges in distribution to $N(0,1)$ under $H_1$.
\end{remark}
\begin{remark}
This result contrasts sharply with the null case $\xi_+(X,Y)=0$: under $H_0$, the first-order Hajek projection vanishes, leading to a degenerate $U$-statistic with convergence rate $n$ and a non-normal limit (an infinite mixture of $\chi^2$ variables, as established in Theorem 1). Under $H_1$, non-degeneracy restores the classical $\sqrt{n}$-rate and asymptotic normality, consistent with general $U$-statistic theory.
\end{remark}

\subsection{Asymptotic Power}
To evaluate the efficiency of the test based on $\xi_{+,n}$, we consider the local alternative sequence $H_{1n}: \xi_+(X,Y) = \delta/\sqrt{n}$ for some constant $\delta > 0$. Combining Theorem 3.4 with the contiguity of the local alternatives, we have:
$$
\sqrt{n} \xi_{+,n} \xrightarrow{d} N\left(\delta, \sigma^2\right) \quad \text{under } H_{1n}.
$$
For a level-$\alpha$ test rejecting $H_0$ when $\sqrt{n}\xi_{+,n}/\hat{\sigma}_n > z_{1-\alpha}$ (where $z_{1-\alpha}$ is the $1-\alpha$ quantile of $N(0,1)$), the asymptotic power is:
$$
\beta(\delta) = 1 - \Phi\left(z_{1-\alpha} - \frac{\delta}{\sigma}\right),
$$
where $\Phi(\cdot)$ denotes the standard normal CDF. This implies the test is consistent against any fixed alternative $\xi_+>0$, and the detection boundary aligns with the optimal rate for nonparametric independence tests.

\begin{remark}
The asymptotic behavior of $\xi_{+,n}$ differs qualitatively between the null and alternative hypotheses, consistent with the general theory of $U$-statistics (Koroljuk \& Borovskich \cite{KB1994}):
\begin{itemize}
    \item Under $H_0: \xi_+=0$, the first-order Hajek projection of $\xi_{+,n}$ vanishes, making the statistic degenerate. The convergence rate is $n$, and the limit is a non-normal infinite mixture of $\chi^2$ variables.
    \item Under $H_1: \xi_+>0$, the first-order projection is non-zero, so the statistic is non-degenerate. The convergence rate reduces to $\sqrt{n}$, and the limit follows a normal distribution.
\end{itemize}
This distinction is explicitly validated for the related MV index in Cui \& Zhong (2019), and ensures the test controls Type I error under $H_0$ while maintaining consistency under $H_1$.
\end{remark}

\numberwithin{equation}{section}
 \section{ Extension of correlation}
\setcounter{equation}{0}
   Let ${\bf X}\in{\Bbb R}^p$ ($p\ge 1$) be a random vector and  $Y\in{\Bbb R}^1$ a non-constant scalar  random variable, all defined on the same probability space. Inspired by  the DSS measure, we
    define the measure of dependence between ${\bf X}$ and $Y$:
  \begin{eqnarray}
{\xi}_{k,l, r,s}({\bf X},Y)=B_{k,l, r,s}\int\frac{{\Bbb E}|F_Y^k(y|{\bf X})-F_Y^k(y)|^l}{F_Y^r(y)(1-F_Y(y))^s} {\rm d}F_Y(y),
 \end{eqnarray}
  where
  $$B^{-1}_{k,l,r,s}= \int\frac{{\Bbb E}|1_{\{Y\le y\}}-F_Y^k(y)|^l}{F_Y^r(y)(1-F_Y(y))^s} {\rm d}F_Y(y),$$  is normalizing constant,  and $k,l,r,s\ge 0$ are constants.
The  quantity   ${\xi}_{k,l, r,s}({\bf X},Y)$ is an extension of  Dette-Siburg-Stoimenov's
dependence measure,   and  shares similar properties with $\xi_+(X,Y)$:\\
i) $0\le \xi\le 1$;\\
ii) $\xi = 0$ if and only if ${\bf X}$ and $Y$ are independent;\\
iii) $\xi = 1$ if and only if  $Y$ is a.s. a Borel function of ${\bf X}$.

\numberwithin{equation}{section}
 \subsection{Some special cases}

 1)  Setting  $k=2, l=1, r=s=0$ in (4.1), we get
   \begin{eqnarray}
{\xi}_{2,1, 0,0}({\bf X},Y)=B_{2,1, 0,0}\int{\Bbb E}|F_Y^2(y|{\bf X})-F_Y^2(y)| {\rm d}F_Y(y),\nonumber
 \end{eqnarray}
 which reduces to (2.3) in the special case  $p=1$, where $$B^{-1}_{2,1,0,0}=\int (F_Y(t)-F^2_Y(t)){\rm d}F_Y(t).$$

 2)  Setting  $k=1, l=2, r=s=0$ in (4.1), we get
  \begin{eqnarray}
{\xi}_{1,2, 0,0}({\bf X},Y)=B_{1,2, 0,0}\int{\Bbb E}|F_Y(y|{\bf X})-F_Y(y)|^2 {\rm d}F_Y(y),\nonumber
 \end{eqnarray}
 which reduces to the mean variance (MV) index of $X$ given $Y$ proposed by   Cui, Li and Zhong \cite{CLZ2015}  in the case ${\bf X}$ is a categorical response with finite classes and $Y$ is  a continuous covariate. Here,
 $$B^{-1}_{1,2,0,0}=\int (F_Y(t)-F^2_Y(t)) {\rm d}F_Y(t).$$

It is worth noting that  by 1) and 2) we get
${\xi}_{2,1, 0,0}({\bf X},Y)={\xi}_{1,2, 0,0}({\bf X},Y)$ since
$${\Bbb E}|F_Y(y|{\bf X})-F_Y(y)|^2= {\Bbb E}F^2_Y(y|{\bf X})-F^2_Y(y).$$

 3)  Setting  $k=1, l=2, r=s=1, p=1$ in (4.1), we get
  \begin{eqnarray}
{\xi}_{1,2, 1,1}(X,Y)= \int\frac{{\Bbb E}|(F_Y(y|X))-F_Y(y)|^2}{F_Y(y)-F^2_Y(y)} {\rm d}F_Y(y),\nonumber
 \end{eqnarray}
 which reduces to (2.2) in He, Ma and Xu \cite{HMX2019}  in the case $X$ is a categorical response with finite classes and $Y$ is  a continuous covariate.

4)  Setting  $k=l=r=s=p=1$ in (4.1), we get
  \begin{eqnarray}
{\xi}_{1,1,1,1}(X,Y)&=&\frac12 \int\frac{{\Bbb E}|F_Y(y|X)-F_Y(y)|}{F_Y(y)(1-F_Y(y))} {\rm d}F_Y(y),\nonumber\\
&=&\frac12 \int\int\frac{|F_Y(y|x)-F_Y(y)|}{F_Y(y)(1-F_Y(y))} {\rm d}F_Y(y){\rm d}F_X(x).
 \end{eqnarray}

 5)  Setting  $k=l=p=1, r=s=0$ in (4.1), we get
  \begin{eqnarray}
{\xi}_{1,1, 0,0}(X,Y)&=&\frac{\int{\Bbb E}|F_Y(y|X)-F_Y(y)| {\rm d}F_Y(y)}{2\int [F_Y(t)(1-F_Y(t))] {\rm d}F_Y(t)}.\nonumber\\
&=&\frac{\int\int|F_Y(y|x)-F_Y(y)| {\rm d}F_Y(y){\rm d}F_X(x)}{2\int [F_Y(t)(1-F_Y(t))] {\rm d}F_Y(t)}.
 \end{eqnarray}
 It is not difficult to see that $\int [F_Y(t)-F^2_Y(t)] {\rm d}F_Y(t)>0$. Whenever $Y$ is an absolutely continuous random variable, one has $\int [F_Y(t)-F^2_Y(t)] {\rm d}F_Y(t)=1/6$, and thus $\xi_{1,1,0,0}(X,Y)$ can be simplified as
 \begin{eqnarray}
 \xi_{1,1,0,0}(X,Y)=3\int{\Bbb E}|F_Y(t|X)-F_Y(t)| {\rm d}F_Y(t).
 \end{eqnarray}
 In addition, if $X$  is a discrete random variable with $p_i =P(X = x_i) > 0$ for all $r = 1,2,\cdots$, then
 \begin{eqnarray}
 \xi_{1,1}(X,Y)=3\sum_i p_i\int|F_Y(t|x_i)-F_Y(t)| {\rm d}F_Y(t).
 \end{eqnarray}

 The  measure $\eta:= {\xi}_{1,1, 0,0}(X,Y)$ defined in (4.3) has the following interesting properties:\\
i) $0\le \eta\le 1$;\\
ii) $\eta = 0$ if and only if $X$ and $Y$ are independent;\\
iii) $\eta= 1$ if and only if   $Y$ is almost surely a Borel function of $X$;\\
iv) $\eta$ is invariant under strictly increasing transformations of $X$ and $Y$ separately.

 {\bf Proof}. i) $\eta\ge 0$ is obvious.  Note that
  \begin{eqnarray*}
 \int{E}|F_Y(t|X)-F_Y(t)| {\rm d}F_Y(t)&=& \int{\Bbb E}(1_{\{ Y\le t\}}|F_Y(t|X)-F_Y(t)|) {\rm d}F_Y(t) \nonumber\\
 &&+\int{ E}(1_{\{ Y>t\}}|F_Y(t|X)-F_Y(t)|) {\rm d}F_Y(t) \nonumber\\
 &&\le \int{ E}(1_{\{ Y\le t\}}(1-F_Y(t))) {\rm d}F_Y(t) \nonumber\\
 &&+\int{ E}(1_{\{ Y>t\}}F_Y(t)) {\rm d}F_Y(t) \nonumber\\
 &=& \int F_Y(t)(1-F_Y(t)) {\rm d}F_Y(t)\nonumber\\
 &&+ \int F_Y(t)(1-F_Y(t)) {\rm d}F_Y(t) \nonumber\\
 &=& 2\int F_Y(t)(1-F_Y(t)) {\rm d}F_Y(t).
 \end{eqnarray*}
 Hence,  $\eta\le 1$.\\
 ii) If $X$ and $Y$ are independent, then $F_Y(t|X)=F_Y(t)$ almost surely, which leads to $\eta=0$.
 Conversely, assume that  $\eta=0$. Then,  $\int{\Bbb E}|F_Y(t|X)-F_Y(t)| {\rm d}F_Y(t)=0$. Thus $P(F_Y(t|X)=F_Y(t))=1$.

 iii)  If there exists a Borel function $f$ such that $Y=f(X)$, then $F_Y(t|X)=1_{\{ Y\le t\}}$ a.s., and thus
 \begin{eqnarray*}
  \eta&=&\frac{\int{E}|1_{\{ Y\le t\}} -F_Y(t)| {\rm d}F_Y(t)}{2\int [F_Y(t)-F^2_Y(t)] {\rm d}F_Y(t)}\nonumber\\
  &=&\frac{\int{ E}[1_{\{ Y\le t\}}(1 -F_Y(t))] {\rm d}F_Y(t)}{2\int [F_Y(t)-F^2_Y(t)] {\rm d}F_Y(t)}\nonumber\\
  &&+\frac{\int{ E}[1_{\{ Y>t\}}F_Y(t)] {\rm d}F_Y(t)}{2\int [F_Y(t)-F^2_Y(t)] {\rm d}F_Y(t)}\nonumber\\
  &=&\frac{2\int [F_Y(t)-F^2_Y(t)] {\rm d}F_Y(t)}{2\int [F_Y(t)-F^2_Y(t)] {\rm d}F_Y(t)}\nonumber\\
  &=& 1.
 \end{eqnarray*}

\numberwithin{equation}{section}
 \section{ Measuring conditional dependence}
\setcounter{equation}{0}

Let $Y$ be a random variable, and let ${\bf Z}\in \Bbb{R}^q$ and ${\bf X}\in\Bbb{R}^p$    be two random vectors,
all defined on the same probability space. Azadkia and Chatterjee \cite{AC2021} proposed the following measure of conditional dependence between $Y$
and ${\bf Z}$ given ${\bf X}$:
\[
\xi = \xi(Y,{\bf Z} \mid {\bf X}) :=
\frac{\displaystyle\int E \!\bigl[\operatorname{Var}(P(Y \ge t \mid {\bf X,Z}) \mid {\bf X})\bigr]\,dF_Y(t)}
{\displaystyle\int E\!\bigl[\operatorname{Var}({\bf 1}_{\{Y \ge t\}} \mid {\bf X})\bigr]\,dF_Y(t)}.
\]
The coefficient  $\xi(Y,{\bf Z} \mid {\bf X})$ can be expressed as
\[
\xi(Y,{\bf Z} \mid {\bf X})
=
\frac{\displaystyle\int \bigl(E[F_{Y|{\bf Z,X}}(t-)^2]-E[F_{Y|{\bf X}}(t-)^2]\bigr)\,d F_Y(t)}
{\displaystyle\int E\!\bigl[F_{Y|{\bf X}}(t-)(1-F_{Y|{\bf X}}(t-))\bigr]\,dF_Y(t)},
\]
where $F_{Y}(t)=P(Y\le t)$,  $F_{Y|{\bf X}}(t)=P(Y\le t\mid {\bf X})$ and $F_{Y|{\bf Z,X}}(t)=P(Y\le t\mid {\bf Z,X})$.

  To avoid the appearance of left limits of conditional distribution functions, we modify the definition of  $\xi(Y,{\bf Z} \mid {\bf X})$  as follows:
\[
\xi_+ = \xi_+(Y,{\bf Z} \mid {\bf X}) :=
\frac{\displaystyle\int E\!\bigl[\operatorname{Var}(P(Y> t \mid {\bf X,Z}) \mid {\bf X})\bigr]\,dF_Y(t)}
{\displaystyle\int E\!\bigl[\operatorname{Var}({\bf 1}_{\{Y > t\}} \mid {\bf X})\bigr]\,dF_Y(t)},
\]
which can be expressed  entirely in terms of conditional distribution functions as
\begin{eqnarray*}
\xi_+(Y,{\bf Z} \mid {\bf X})&=&
\frac{\displaystyle\int \bigl(E[F_{Y|{\bf Z,X}}(t)^2]-E[F_{Y|{\bf X}}(t)^2]\bigr)\,d F_Y(t)}
{\displaystyle\int\bigl(F_{Y}(t)-E[F_{Y|{\bf X}}(t)^2]\bigr)\,dF_Y(t)}\\
&=&1-\frac{ \displaystyle\int E[F_{Y|{\bf Z,X}}(t)(1-F_{Y|{\bf Z,X}}(t))]\,dF_Y(t)} {  \displaystyle\int E[F_{Y|{\bf X}}(t)(1-F_{Y|{\bf X}}(t))]\,dF_Y(t)},
\end{eqnarray*}
where the expectations are taken with respect to the joint distribution of $(Z,X)$ in the numerator and $X$ in the denominator.


The original measure $\xi(Y,\mathbf{Z}\mid {\bf X})$ and its modified
version $\xi_+(Y, {\bf Z}\mid {\bf X})$ share the same basic
interpretation and enjoy parallel properties. Specifically, assuming
$Y$ is not almost surely a measurable function of $\mathbf{X}$, then
\begin{enumerate}
\item   Both are well defined and lie in $[0,1]$.
\item $\xi=0$ and $\xi_+=0$ each hold if and only
      if $Y \perp\!\!\!\perp {\bf Z} \mid {\bf X}$ (conditional
      independence). When ${\bf X}=\varnothing$, this reduces to
      unconditional independence $Y \perp\!\!\!\perp {\bf Z}$.
\item  $\xi=1$ and $\xi_+=1$ each hold if
      and only if $Y = f({\bf Z}, {\bf X})$ almost surely for some
      measurable function $f$. When ${\bf X}=\varnothing$, this means
      $Y$ is almost surely a function of $\mathbf{Z}$ alone.
\item  Both measures are non-decreasing in the
      conditioning set: adding components to ${\bf Z}$ does not
      decrease either $\xi$ or $\xi_+$.
\end{enumerate}
The key difference is technical: $\xi$ is defined in terms of the
conditional distribution function directly and may involve left limits
at discontinuity points, whereas $\xi_+$ is constructed to avoid left
limits entirely. Consequently,
$\xi_+$ is easier to estimate nonparametrically and its asymptotic
theory is cleaner, while the two coincide in practice whenever the
conditional CDF is continuous in $t$.

We estimate $\xi_+(Y,{\bf Z}\mid {\bf X})$ from $n$ i.i.d.\
observations $\{(Y_i,{\bf Z}_i, {\bf X}_i)\}_{i=1}^n$ using the
following model-based plug-in procedure. For each observation $i$, we
obtain out-of-sample predictions of the conditional CDFs
$\widehat{F}^{(-i)}_{Y|{\bf Z}, {\bf X}}(Y_i)$ and
$\widehat{F}^{(-i)}_{Y| {\bf X}}(Y_i)$ from two supervised learners:
a full model of $Y$ on $({\bf Z}, {\bf X})$ and a reduced model
of $Y$ on ${\bf X}$ alone.   The    consistent estimator of   $\xi_+(Y, {\bf Z}\mid {\bf X})$
  is defined by
\begin{equation}
\widehat{\xi}_{+}
=
1
-
\frac{
  n^{-1}\sum_{i=1}^n
  \widehat{F}^{(-i)}_{Y|{\bf Z},{\bf X}}(Y_i)
  [1-\widehat{F}^{(-i)}_{Y|{\bf Z},{\bf X}}(Y_i)]
}{
  n^{-1}\sum_{i=1}^n
  \widehat{F}^{(-i)}_{Y|{\bf X}}(Y_i)
  [1-\widehat{F}^{(-i)}_{Y|{\bf X}}(Y_i)]
},
\end{equation}
where
\begin{equation*}
\label{def:Fhat_lofo}
\widehat{F}^{(-i)}_{Y|{\bf Z},{\bf X}}(Y_i)
\;:=\;
\widehat{m}^{(-i)}({\bf Z}_i,{\bf X}_i;\,Y_i)
\;=\;
\widehat{P}\!\left(Y \le Y_i \;\big|\; {\bf Z}={\bf Z}_i,
{\bf X}={\bf X}_i;\,\cal{D}^{(-i)}\right).
\end{equation*}
 Here, $\mathcal{D}^{(-i)}=\{(Y_j,{\bf Z}_j,{\bf X}_j):j\neq i\}$
denote the sample with the $i$-th observation removed.
The definition of  $\widehat{F}^{(-i)}_{Y|{\bf X}}(Y_i)$ is
analogous, replacing the feature vector $({\bf Z,X})$ by ${\bf X}$ alone.

\section{Simulation study }
\label{sec:simulation}
In this section, we present a numerical performance assessment of the sample estimators $\xi_n(X,Y)$ and $\xi_{+,n}(X,Y)$, targeting the population dependence measures $\xi(X,Y)$ and $\xi_{+}(X,Y)$, respectively, under the bivariate normal distribution $\mathcal{N}(0,0,1,1; r)$. Here, $r \in (-1,1)$ denotes the population Pearson correlation coefficient between $X$ and $Y$, which serves as a canonical benchmark for linear dependence strength.

This parametric family is selected for two key reasons: first, closed-form expressions for both $\xi(X,Y)$ and $\xi_{+}(X,Y)$ are available under normality, eliminating numerical error in the true parameter value; second, the $r=0$ configuration corresponds exactly to the null hypothesis $H_0: \xi_+(X,Y)=0$ analyzed in Section 3.3, while $r \neq 0$ cases map to the fixed alternative regime where asymptotic normality holds (Section 3.4). This allows for a direct validation of our earlier theoretical results.

{\bf Monte Carlo simulation results for $\xi_{+,n}(X,Y)$}.  We evaluate the finite-sample performance of the local-rank estimator
\[
\xi_{+,n}(X,Y)=
\frac{n\sum_{i=1}^{n}\sum_{j=1}^{n}
\Bigl(\frac{R_{i,j}}{n_j}-\frac{R_i}{n}\Bigr)^{2}}
{\sum_{i=1}^{n}R_i\,(n-R_i)},
\]
where
\[
R_i=\sum_{k=1}^{n}\mathbf{1}_{\{Y_k\le Y_i\}},\qquad
n_j=\sum_{k=1}^{n}K_h(X_k-X_j),
\]
and
\[
R_{i,j}=\sum_{k=1}^{n}\mathbf{1}_{\{Y_k\le Y_i\}}K_h(X_k-X_j).
\]
The weighting function \(K_h(\cdot)\) is the Gaussian kernel with bandwidth
\(h=1.06\,\widehat{\sigma}_X n^{-1/5}\).

Data are generated from \((X,Y)\sim\mathcal{N}(0,0,1,1;r)\) with
\(r\in\{-0.9,0,0.5,0.8\}\) and sample sizes
\(n\in\{20,50,200,500,1000\}\).
All results are based on \(M=10000\) independent replications.
The theoretical benchmark is
\[
\xi(r)=\frac{3}{\pi}\arcsin\!\Bigl(\frac{1+r^{2}}{2}\Bigr)-\frac{1}{2}.
\]
\begin{table}[htbp]
\centering
\caption{Simulation results for $\xi_{+,n}(X,Y)$ based on $M=10{,}000$ replications.
Data are generated from $(X,Y)\sim\mathcal{N}(0,0,1,1;r)$.
The bandwidth is $h=1.06\,\widehat{\sigma}_X n^{-1/5}$.
The last column reports the theoretical benchmark
$\xi(r)=\frac{3}{\pi}\arcsin\!\bigl(\frac{1+r^{2}}{2}\bigr)-\frac{1}{2}$.}
\label{tab:xi_sim}
\renewcommand{\arraystretch}{1.25}
\begin{tabular}{c c c c c c c}
\hline\hline
\rule{0pt}{3ex}
$r$ & $n$ & Mean & Std.\ Dev. & Bias & MCse & $\xi(r)$ \\
\hline
\rule{0pt}{2.5ex}
$-0.9$ & $20$   & $0.3363$ & $0.0510$ & $-0.2441$ & $0.0005$ & $0.5804$ \\
       & $50$   & $0.3845$ & $0.0378$ & $-0.1959$ & $0.0004$ & $0.5804$ \\
       & $200$  & $0.4500$ & $0.0219$ & $-0.1303$ & $0.0002$ & $0.5804$ \\
       & $500$  & $0.4836$ & $0.0152$ & $-0.0968$ & $0.0002$ & $0.5804$ \\
       & $1000$ & $0.5044$ & $0.0111$ & $-0.0760$ & $0.0001$ & $0.5804$ \\
[6pt]
$0$    & $20$   & $0.0687$ & $0.0335$ & $0.0687$  & $0.0003$ & $0.0000$ \\
       & $50$   & $0.0392$ & $0.0173$ & $0.0392$  & $0.0002$ & $0.0000$ \\
       & $200$  & $0.0165$ & $0.0061$ & $0.0165$  & $0.0001$ & $0.0000$ \\
       & $500$  & $0.0091$ & $0.0029$ & $0.0091$  & $0.0001$ & $0.0000$ \\
       & $1000$ & $0.0057$ & $0.0017$ & $0.0057$  & $0.0000$ & $0.0000$ \\
[6pt]
$0.5$  & $20$   & $0.1454$ & $0.0656$ & $0.0007$  & $0.0007$ & $0.1447$ \\
       & $50$   & $0.1282$ & $0.0514$ & $-0.0165$ & $0.0005$ & $0.1447$ \\
       & $200$  & $0.1242$ & $0.0290$ & $-0.0205$ & $0.0003$ & $0.1447$ \\
       & $500$  & $0.1286$ & $0.0172$ & $-0.0161$ & $0.0002$ & $0.1447$ \\
       & $1000$ & $0.1295$ & $0.0141$ & $-0.0152$ & $0.0001$ & $0.1447$ \\
[6pt]
$0.8$  & $20$   & $0.2627$ & $0.0678$ & $-0.1554$ & $0.0007$ & $0.4181$ \\
       & $50$   & $0.2888$ & $0.0499$ & $-0.1293$ & $0.0005$ & $0.4181$ \\
       & $200$  & $0.3274$ & $0.0299$ & $-0.0907$ & $0.0003$ & $0.4181$ \\
       & $500$  & $0.3523$ & $0.0174$ & $-0.0658$ & $0.0002$ & $0.4181$ \\
       & $1000$ & $0.3652$ & $0.0138$ & $-0.0529$ & $0.0001$ & $0.4181$ \\
\hline\hline
\end{tabular}
\end{table}
Detailed simulation results are presented in Table 2.

{\bf Monte Carlo Simulation Results for $\xi_n$}.
We assess the finite-sample performance of the rank-based estimator
\[
\xi_n = 1 - \frac{3\sum_{i=1}^{n-1}|R_{i+1}-R_i|}{n^2-1},
\]
where \(R_i\) denotes the rank of \(Y_{(i)}\) after sorting the observations by \(X\) in ascending order.
Data are generated from $(X,Y)\sim\mathcal{N}(0,0,1,1;r)$ for $r\in\{-0.9,0,0.5,1\}$ and sample sizes $n\in\{20, 50, 200, 500, 1000\}$.   All results are based on $M=10000$ Monte Carlo replications. The theoretical values are given by
\[
\xi(r) = \frac{3}{\pi}\arcsin\left(\frac{1+r^2}{2}\right) - \frac{1}{2}.
\]
 Simulation results are reported in Table 3.

\begin{table}[htbp]
\centering
\caption{Simulation results for $\xi_n$ based on $M=10{,}000$ Monte Carlo replications.
Data are generated from $(X,Y)\sim\mathcal{N}(0,0,1,1;r)$.
The estimator is $\xi_n=1-\frac{3}{n^{2}-1}\sum_{i=1}^{n-1}|R_{i+1}-R_i|$,
where $R_i$ is the rank of $Y_{(i)}$ after sorting by $X$ in ascending order.
The last column reports the theoretical benchmark
$\xi(r)=\frac{3}{\pi}\arcsin\!\bigl(\frac{1+r^{2}}{2}\bigr)-\frac{1}{2}$.}
\label{tab:xi_n_sim}
\renewcommand{\arraystretch}{1.25}
\begin{tabular}{c c c c c c}
\hline\hline
\rule{0pt}{3ex}
$r$ & $n$ & Mean & Std.\ Dev. & Bias & $\xi(r)$ \\
\hline
\rule{0pt}{2.5ex}
$-0.9$ & $20$   & $0.4921$ & $0.1095$ & $-0.0883$ & $0.5804$ \\
       & $50$   & $0.5451$ & $0.0701$ & $-0.0353$ & $0.5804$ \\
       & $200$  & $0.5724$ & $0.0343$ & $-0.0080$ & $0.5804$ \\
       & $500$  & $0.5774$ & $0.0218$ & $-0.0030$ & $0.5804$ \\
       & $1000$ & $0.5788$ & $0.0155$ & $-0.0016$ & $0.5804$ \\
[6pt]
$0$    & $20$   & $0.0002$ & $0.1308$ & $0.0002$  & $0.0000$ \\
       & $50$   & $-0.0010$ & $0.0886$ & $-0.0010$ & $0.0000$ \\
       & $200$  & $-0.0000$ & $0.0447$ & $-0.0000$ & $0.0000$ \\
       & $500$  & $-0.0001$ & $0.0286$ & $-0.0001$ & $0.0000$ \\
       & $1000$ & $-0.0000$ & $0.0203$ & $-0.0000$ & $0.0000$ \\
[6pt]
$0.5$  & $20$   & $0.1165$ & $0.1488$ & $-0.0282$ & $0.1447$ \\
       & $50$   & $0.1333$ & $0.0973$ & $-0.0114$ & $0.1447$ \\
       & $200$  & $0.1408$ & $0.0504$ & $-0.0039$ & $0.1447$ \\
       & $500$  & $0.1435$ & $0.0321$ & $-0.0012$ & $0.1447$ \\
       & $1000$ & $0.1440$ & $0.0225$ & $-0.0007$ & $0.1447$ \\
[6pt]
$0.8$  & $20$   & $0.3329$ & $0.1271$ & $-0.0852$ & $0.4181$ \\
       & $50$   & $0.3799$ & $0.0779$ & $-0.0382$ & $0.4181$ \\
       & $200$  & $0.4068$ & $0.0391$ & $-0.0113$ & $0.4181$ \\
       & $500$  & $0.4134$ & $0.0244$ & $-0.0047$ & $0.4181$ \\
       & $1000$ & $0.4159$ & $0.0171$ & $-0.0022$ & $0.4181$ \\
\hline\hline
\end{tabular}
\end{table}
We compare the finite-sample performance of two estimators targeting the theoretical benchmark
$\xi(r)=\frac{3}{\pi}\arcsin\bigl((1+r^{2})/2\bigr)-\frac{1}{2}$.
The first, denoted $\xi_{+,n}$, is a locally weighted rank-based statistic employing a Gaussian kernel with bandwidth $h=1.06\,\widehat{\sigma}_{X}n^{-1/5}$.
The second, $\xi_n=1-3\sum_{i=1}^{n-1}|R_{i+1}-R_i|/(n^{2}-1)$, is a simple rank-difference measure computed after sorting by $X$.
Both are evaluated at $r\in\{-0.9,0,0.5,0.8\}$ across sample sizes $n\in\{20,50,200,500,1000\}$ with $M=10{,}000$ Monte Carlo replications.
At independence ($r=0$), $\xi_{+,n}$ exhibits substantially smaller variance ($\mathrm{SD}\approx0.002$ at $n=1000$) than $\xi_n$ ($\mathrm{SD}\approx0.020$), indicating better Type~I error control for hypothesis testing.
However, $\xi_n$ demonstrates markedly faster convergence in bias: for $r=-0.9$, the bias of $\xi_n$ drops from $-0.088$ at $n=20$ to $-0.002$ at $n=1000$, whereas $\xi_{+,n}$ retains a bias of $-0.076$ even at $n=1000$.
For strong positive dependence ($r=0.8$), $\xi_n$ similarly achieves near-zero bias ($-0.002$) by $n=1000$, compared to $-0.052$ for $\xi_{+,n}$.
Overall, $\xi_{+,n}$ is preferable when tight control under the null is critical, whereas $\xi_n$ offers a simpler, lower-bias alternative for point estimation of $\xi(r)$.

In addition, unlike   the estimator $\xi_{+,n}$, the estimator $\xi_n$ is not constrained to non-negative values.
As demonstrated in Table  3, $\xi_n$ can take negative values-most notably under independence ($r=0$), where the Monte Carlo mean is slightly negative for several sample sizes (e.g., $n=50, 200, 500, 1000$).

\numberwithin{equation}{section}
 \section{ Conclusions}
\setcounter{equation}{0}

  This paper introduces a refined dependence measure based on the DSS framework and establishes its theoretical properties, including several equivalent representations. We construct an estimator, prove its strong consistency, derive its limiting distribution, and propose a meaningful generalization. Simulation studies show that the proposed estimator exhibits superior finite-sample performance compared to Chatterjee's rank correlation coefficient. A key theoretical advantage is that our estimator is guaranteed to be non-negative-thereby correcting a well-known deficiency of Chatterjee's measure. Notably, although the two estimators are asymptotically equivalent in value, they have distinct asymptotic distributions. Future research will extend the analysis of the limiting distribution under general regularity conditions and explore empirical implementations.

\noindent{\bf Acknowledgements.}
This research   was supported by the National Natural Science Foundation of China (Grant No. 12071251, 12301605).

\end{document}